\documentclass{amsart}
\usepackage{amssymb,latexsym} 
\usepackage{color}

 \newtheorem{thm}{Theorem}[section]
 \newtheorem{cor}[thm]{Corollary}
 \newtheorem{lem}[thm]{Lemma}
 \newtheorem{prop}[thm]{Proposition}
 \theoremstyle{definition}
 
 \newtheorem{rem}[thm]{Remark}
 \theoremstyle{definition}
 \newtheorem{ex}[thm]{Example}
\newtheorem{nota}[thm]{Notation}
 \newtheorem*{thmNONUMBER}{Theorem}

 \newcommand{\CC}{\mathbb{C}}
 
  \newcommand{\PP}{\mathbb{P}}
    \newcommand{\VV}{\mathcal{V}}

\begin{document}

\title[Expressing a General Form as a Sum of Determinants]
{Expressing a General Form as a Sum of Determinants}

\author[L. Chiantini ]{Luca Chiantini}
\address[L. Chiantini]{Universit\`a degli Studi di Siena, 
Dipartimento di Scienze Matematiche e Informatiche, 
Pian dei Mantellini, 44, 53100 Siena, Italy.}
\email{luca.chiantini@unisi.it}

\author[Anthony V. Geramita]{Anthony V. Geramita}
\address[A.V. Geramita]{DIMA - Dipartimento di Matematica  
Universit\`{a} di Genova, 16129 Genova, Italy. Department
of Mathematics, Queen's University, Kingston, Ontario
Canada}

\email{anthony.geramita@gmail.com}


\begin{abstract} 
{Let $A= (a_{ij})$ be a non-negative integer $k\times k$ matrix.  $A$ is a homogeneous matrix if $a_{ij} + a_{kl}=a_{il} + a_{kj}$ for any choice of the four indexes.

\noindent We ask: If $A$ is a homogeneous matrix and if $F$ is a form in $\CC[x_1, \dots x_n]$ with $deg(F) = {\rm trace}(A)$,  what is the least integer, $s(A)$, so that $F = detM_1 + \cdots + detM_{s(A)}$, where the $M_i = (F^i_{lm})$ are $k\times k$ matrices of forms and $deg F^i_{lm} =
a_{lm}$ for every $1\leq i \leq s(A)$?

We consider this problem for $n\geq 4$ and we prove that $s(A) \leq k^{n-3}$ and $s(A) <k^{n-3}$ in infinitely many cases.  However $s(A) = k^{n-3}$ when the integers in $A$ are large with respect to $k$.}

\end{abstract}


\date{\today}
\maketitle

\section*{Introduction}

{Let $F \in \CC[x_1, \ldots , x_n]$ be a general form and $A = (a_{ij})$  a square integer homogeneous matrix with the trace of $A$ (${\rm{tr}}(A)$)  equal to the degree of $F$ ($\deg F$).  In this paper we study representations of $F$ as a sum of determinants of matrices of type $M = (F_{ij})$ where $\deg F_{ij} = a_{ij}$.  

In case the number of variables is two then {\it any} form $F$ in $\CC[x_1,x_2]$ of degree $d$ decomposes as a product of linear forms.  It follows that if $A$ is any square homogeneous matrix of integers with no 
 negative entries and with ${\rm tr}(A) = d$, then $F$ is the determinant of a diagonal matrix whose degree matrix is $A$. 

In the case of 3 variables, the problem was considered classically by the great American mathematician L.E. Dickson (see \cite{D21}), who proved that a general form of degree $d$ is the determinant of a $d\times d$ matrix of linear forms.  The recent paper \cite{CM12}, of J. Migliore and the first author, generalizes this fact. Namely, for any given square homogeneous matrix of integers $A$ having trace $d$, there is a simple necessary and sufficient condition on $A$ which tells us when a general form of degree $d$ in three variables is the determinant of a matrix of forms whose degree matrix is $A$. 

Thus the case of 4 variables is the first non-trivial case not yet considered.
 We address the problem in the present paper. 
We prove our Main Theorem  (see below)   in the special case involving 
general forms in 4 variables 
and non-negative integer matrices $A$, in $\S 2$ .  
The proof for greater than 4 variables is in $\S 4$. }

This problem, of clear algebraic and geometric flavour, turns out to also 
have an interesting connection with some
applications in control theory. 
Indeed, if the algebraic boundary of a region $\Theta$
in the plane or in space, is described by the determinant
of a matrix of linear forms, then the study of systems
of matrix inequalities, whose domain is
$\Theta$, can be considerably simplified. We refer to the papers
\cite{V89} and \cite{HL12}, for an account of this theory.

 When the number of variables is bigger than $3$ one cannot
 hope to describe a general form of degree $d$ with just one determinant.
 For instance, it is a standard fact that a general
 form $F$ in four variables, of degree at least $4$, cannot
 be the determinant of a matrix of linear forms. In fact, if we 
 delete one row of the matrix, one sees that
 the surface $F=0$ should contain a curve cut by hypersurfaces of degree $d-1$.
This is impossible for general surfaces in $\PP^3$ since
 the celebrated Noether-Lefschetz Theorem prevents a general
 surface of degree $d>3$ from containing curves cut by surfaces of degree
 $d-1$.
 
As a consequence, we are led  to the following, quite natural,
 question: for a general form $F$ of degree $d$, and a
  given {\it homogeneous} square matrix of integers $A$, with
  degree (= trace) $d$, how many matrices of forms, with
  degree matrix $A$, are necessary so that $F$ is the
  {\it sum} of their determinants?

In a previous paper (\cite{CG12}), we showed that a general
form in four variables is the sum of two determinants 
of  $2\times 2$ matrices with given degree matrices.

When the size of the degree matrix $A$ grows, one cannot hope
to obtain a similar result, with the sum of just two determinants.
This is clear from a standard geometrical interpretation
of the problem.  The interpretation is based on the study
of secant varieties using the classical Terracini 
Lemma. Let us recall a standard construction, already used in 
\cite{CCG08} and in \cite{CG12}.   

\begin{ex}\label{basic}
Inside the projective space $\PP^N$, which 
parametrizes forms of degree $d$ in four variables (up to scalar
multiplication), the 
set of points representing forms which are the determinant
of a $k\times k$ matrix, whose degree matrix is fixed,
is dense in a projective subvariety $V$. 
Our question can be rephrased by asking: what is the
minimal $s$ such that a general point of $\PP^N$ is spanned by
$s$ points of $V$. In classical Algebraic Geometry,
(the closure of) the set of points spanned by $s$ points of
$V$, is called the {\it $s$-th secant variety} $S^s(V)$ of $V$.
Thus, we look for the minimal $s$ such that $S^s(V)=\PP^N$.

At a general point
$F=\det(G)\in V$, the tangent space to $V$ at $F$ corresponds
to forms of degree $d$ in the ideal $J$, generated by the submaximal minors of $G$.
If the matrix $A$ is $k\times k$, with all entries equal to $a$
(so that $d=ak$), then 
$J$ is generated by $k^2$ forms of degree $a(k-1)$.

By the celebrated Terracini Lemma, 
the tangent space at a general point $F$ of the
$s$-th secant variety is spanned by  $s$ secant spaces 
at the points $G_i\in V$, $i=1,\dots, s$, such that $F=\sum G_i$.
 
Thus, we want to know the minimal $s$ such that, for general matrices
$G_1,\dots, G_s\in V$ with degree matrix $A$, the ideal $I$, generated
by {\it all} their submaximal minors, coincides with the polynomial ring 
$R=\mathbb C[x,y,z,t]$, in degree $d$.

Just computing the dimensions as vector spaces, we see that
\begin{equation}\label{sharp}
\dim I_d\leq k^2s \dim R_a = k^2s a^3/6 + o(a^3)
\end{equation}
while the dimension of $R_d$ is $a^3k^3/6+o(a^3)$.

So, it is immediate to see that, at least when $a$ grows, 
if $I_d=R_d$ then $s$ must be asymptotically equal to $k$.
\end{ex}

We show that the bound of the previous rough estimate,
is always attained. Namely, we prove (see Theorem \ref{main} below):

\begin{thmNONUMBER} ({\bf Main})
Let $A$ be a homogeneous $k\times k$ matrix of \emph{non--negative} 
integers, with  $tr(A) = d$. Then a general form of degree $d$ in $4$
variables is the sum of $k$ determinants of matrices of forms each 
with degree matrix $A$. 
\end{thmNONUMBER}

The proof is based on an algebraic analysis of the ideal
generated by submaximal determinants. Essentially, we use
induction on the degree of $A$. A fundamental point in the proof is 
the fact that, by the main result of \cite{CM12}, the quotient $S$ 
of the polynomial ring $R$, by the
ideal generated by many submaximal minors, satisfies
a sort of {\it weak Lefschetz property}: multiplication by a general
linear form has maximal rank, in degrees close to $d$.

We notice that our result can also be interpreted as a result
for general surfaces with given degree $d$ in the projective space
$\PP^3$. The Hilbert--Burch Theorem shows that homogeneous
$(k-1)\times k$ matrices of forms determine the resolution
of ideals of curves which are {\it arithmetically Cohen--Macaulay}.

We first extend the idea of {\it trace} to matrices of size $(k-1)\times k$ by 
defining the trace of such a matrix to be the maximal trace of any square $(k-1)\times (k-1)$
submatrix.  We can now state our result in terms of surfaces containing 
curves of given type. 

\begin{cor}\label{mainc}
Let $A'$ be a homogeneous $(k-1)\times k$ matrix of \emph{non--negative} 
integers. Then a general surface of degree $d\geq tr(A') $ in $\PP^3$
is contained in a linear system generated by $k$ surfaces, each of 
 which contains an arithmetically Cohen--Macaulay curve
whose Hilbert--Burch matrix has degree matrix equal to $A'$.
\end{cor}

As we showed in the previous example, the conclusion of our Main Theorem cannot be 
 improved for certain matrices $A$ (see also Example \ref{3.5}). However,
we know that in some specific cases 
(e.g. when all the entries of $A$ are $1$'s, so that
we are looking at determinants of matrices of linear forms) the
number of determinants needed to write a general form
can be smaller than our bound $k$. See Remark \ref{sharprem} and \S \ref{impro}
for a discussion.
The problem of finding a {\it sharp}
 bound for the number of determinants needed to express
 a general form of small degree is still open.

 The extension of our Main Theorem to the case of $n > 4$ variables is in $\S 4$.

\section*{Preliminaries}

We work in the ring $R = \CC[x,y,z,t]$, i.e. the polynomial ring in 4 variables
with coefficients in the complex numbers.  By {\it quaternary form}, we mean any
homogeneous polynomial in $R$ and by $R_n$ we mean the vector space
of (quaternary) forms of degree $n$ in $R$.

By abuse of notation  we will often indicate with the same symbol $F$, 
both a form $F\in \CC[x,y,z,w]$ and the surfaces defined by {the} equation $F=0$.

Fix a degree $n$. The space $R_n$ of forms of degree $n$
has an associated projective space $\PP^N$ with 
$$N:=N(n)= \binom{n+3}3 -1.$$ 

For any choice of integers $a_{ij}$, $1\leq i,j\leq k$,
consider the numerical $k\times k$ matrix $A=(a_{ij})$.

We will say that a $k\times k$ matrix $M=(F_{ij})$, whose entries are 
(quaternary) forms, has {\it degree matrix $A$} if, for all
$i,j$, we have $\deg(F_{ij})=a_{ij}$. In this case, we will also write 
that $A=\partial M$.

Notice that when, for some $i,j$, we have $F_{ij}=0$, there are several 
possible degree matrices for $M$, since the zero polynomial is considered to have any degree. 
  \smallskip

Notice that the set of all matrices of forms whose degree matrix
is a fixed $A$, defines a vector space whose dimension 
is  $\sum \dim(R_{a_{ij}})$. From the geometrical point of
view, however, we will consider this set as the {\it product}
of projective spaces
$$ \VV(A) = \PP^{r_{11}}\times\cdots\times\PP^{r_{kk}}$$
where $r_{ij}=-1+\dim(R_{a_{ij}})$. 
\smallskip

We say that the numerical matrix $A$ is {\it homogeneous} when,
for any choice of the indexes $i,j,l,m$, we have
$$ a_{ij}+a_{lm} = a_{im}+a_{lj}.$$

All submatrices of a homogeneous matrix are homogeneous.

If a square matrix of forms $M$ has a homogeneous degree matrix, then
the determinant of $M$ is a homogeneous form. 
The degree of the determinant is the sum of the numbers on the main 
diagonal of $A$, i.e. $tr(A)$. This number is called the {\it degree} of the homogeneous 
square matrix $A$. 
\smallskip

In the {\it projective} space $\PP^N$, which parametrizes
all forms of degree $n$, we have the subset $U$ of all the forms
which are the determinant of a matrix of forms whose degree
matrix is a given $A$.  This set is a quasi-projective variety,
since it corresponds to the image of the map $\VV(A)\to\PP^N$,
which sends every matrix to its determinant (it is undefined
when the determinant is the polynomial $0$).

We will denote by $V$ the closure of the image of this map. 
As explained in the introduction, a general (quaternary) form $F$ 
of degree at least $4$ cannot be the determinant of a matrix of forms.
Thus, $V$ is not equal to $\PP^N$ when the degree of $A$
is at least $4$.\smallskip

In view of Terracini's Lemma (mentioned in the Introduction)
 we need to characterize the tangent space
to $V$ at a general point $F$.

\begin{prop}\label{tgkxk} Let $F$ be a general element in $V$, 
$F=\det M$, where $M=(F_{ij})$ is a $k\times k$ matrix of forms,
whose degree matrix is $A$.

Then, the tangent space to $V$ at $F$ coincides with
the subspace of $R_n/\langle F \rangle$, generated by the classes
of the forms of degree $n$ in the ideal $\langle F, M_{ij}\rangle$,
where the $M_{ij}$ are the submaximal minors of the matrix $M$. 
\end{prop}
\begin{proof}
This is just a direct computation. Namely, over the ring of dual
numbers $\CC[\epsilon]$ we want to find when the form $F+\epsilon\ G$
is the determinant of a matrix $(F_{ij}+\epsilon G_{ij})$,
where $\deg(G_{ij})=a_{ij}$.

A simple computation shows that this happens exactly when
$G$ sits in the ideal generated by $F$ and the $M_{ij}$'s. 
\end{proof}

\begin{rem}\label{reduc}  It follows immediately from the 
previous propositions, and Terracini's lemma, that:

- a general form of degree $n$ is the sum of $s$  
determinants of $k\times k$ matrices, all having degree matrix $A$ 
(i.e. the span of $s$ general tangent spaces to $V$
is the whole space $\PP^N$),

if and only if
  
- for a general choice of $s$ matrices of forms $M_1,\dots M_s$, of
type $k\times k$, with $\partial M_i=A$ for all $i$, the ideal generated 
by {\it all} the submaximal minors of all the $M_i$'s coincides, in degree $n$, with
the whole space $R_n$.
\end{rem}

\section{Some lemmas about $(k-1)\times k$ matrices of ternary forms}

Let $A'$ be a non-negative integer homogeneous matrix. By performing 
permutations of the rows and the columns of $A'$ we can always assume
that the integers in any row are
increasing as we move to the right and that the integers in any column 
are increasing as we go from bottom to top.  A non-negative integer
 homogeneous matrix whose rows and columns
satisfy the condition just described will be called {\it ordered}. Recall that
we defined the {\it trace} of a  (not necessarily square ) homogeneous matrix to be
the maximum of the traces of its square submatrices. 

In this section we collect results about the ideal generated by
the maximal minors of some
$(k-1)\times k$ non-negative integer homogeneous matrices of ternary forms,
with given degree matrices.

Let $A'=(a'_{ij})$ be such a $(k-1)\times k$ non-negative {\bf ordered}  
integer homogeneous matrix.
  Notice that the trace of $A'$ is equal to
 $$ tr(A')= a_{12}+a_{23}+\dots +a_{k-1\ k}.$$

\begin{rem}\label{HB}
Let $A'=(a_{ij})$ be as above. We will denote by $T(A')$ the number
$$T(A')=tr(A')+a_{11} = a_{11}+ a_{12}+a_{23}+\dots +a_{k-1\ k}.$$
 
 The number $T(A')$ has the following property: for a {\it general} matrix $G$
 of ternary forms, with $\partial G=A'$,  the Hilbert--Burch
theorem implies that the maximal minors of $G$
generate, in the ring $R':=\CC[x,y,z]$, the homogeneous ideal 
$I_{k-1}(G)$ of a set of points $Z\subset \PP^2$
(see the paper \cite{CGO88}, to which we refer for facts about
the Hilbert--Burch matrices of ternary forms).

The Betti numbers of a minimal free resolution  of $I_{k-1}(G)$ are fixed
by the degree matrix $A'$. The number $T(A')$ is exactly the maximal degree
of a syzygy appearing in the resolution of $I_{k-1}(G)$.

It is well known that the Hilbert function of $R'/I_{k-1}(G)$ is equal to the number 
of points in $Z$  for all degrees $n\geq T(A')-2$.  

Moreover,  a general linear form $L$ in $R'$ represents a line in $\PP^2$ 
which meets no point of $Z$. Thus the multiplication by $L$
is an isomorphism
$$ (R'/I_{k-1}(G))_{d-1}\to (R'/I_{k-1}(G))_d$$
whenever the degree $d$ is at least $T(A')-1$.  It follows 
that for any ideal $J\supset I_{k-1}(G)$, multiplication by 
a general linear form $L$
gives a surjective map $(R/J)_{d-1}\to (R/J)_d$ for $d\geq T(A')-1$.
\end{rem}

\begin{ex}  Let
$$ A'=\begin{pmatrix} 5 & 6 & 8 & 9\\ 5 & 6 & 8 & 9 \\ 
2 & 3 & 5 & 6\end{pmatrix}
$$
Since $tr(A') = 20$ we have $T(A')=25$.

For a {\it general} matrix $G$
 of ternary forms, with $\partial G=A'$ and $L$ a general linear form we thus have that the map
 $$ (R'/I_{k-1}(G))_{d-1}\to (R'/I_{k-1}(G))_d$$
  induced by $L$, is an isomorphism as soon as $d\geq 24$.
\end{ex}

We are now ready for our main Lemma.
In order to present its proof in a reasonable fashion we need some other 
pieces of notation. These extend the notation introduced in the previous section.

\begin{nota}
Let $B_1, B_2$ be ordered non-negative integer homogeneous matrices of size $k-1 \times k$.
We will say that 
$$ \mbox{ condition } M_s(B_1^j, B_2^{k-j}) \mbox{ holds }
$$
if for a general choice of $k$ matrices of ternary forms $G_1,\dots, G_k$
with $\partial G_i=B_1$ for all $i\leq j$ and $\partial G_i=B_2$ for all $i> j$, 
the ideal generated by all the
maximal minors $I_{k-1}(G_1)+\dots +I_{k-1}(G_k)$ coincides with
the ring $R':=\CC[x,y,z]$ in degree $s$.

When  $B_1=B_2$, we will write that condition
$M_s(B_1^k)$ holds.
\end{nota}

\begin{lem}  \label{tech}
With the previous notation,  condition
$M_{T(A')}((A')^k) $ holds. I.e. for a general choice of matrices of ternary forms 
$G_1,\dots, G_k$ with $\partial G_i=A'$ for all $i$, then the ideal generated by all the
maximal minors of the $G_i$ i.e. $I_{k-1}(G_1)+\dots +I_{k-1}(G_k)$ coincides with $R'$ in degree $T(A')$.
\end{lem}
\begin{proof}

It's enough to exhibit $k$ matrices with the desired property which, by semi-continuity, implies the result 
for a generic choice.

First assume that all the rows of $A'$ are equal. 
In this case, the Lemma follows from the main result of \cite{CM12}.
To see this, we add a new row to $A'$, equal to all the 
other rows of $A'$.  
We get a square $k\times k$ ordered homogeneous matrix of non-negative integers,
which we denote by $A$. Notice that $tr(A) = T(A')$.

By \cite{CM12}, Theorem 3.6, we know that a general ternary form
of degree equal to the trace of $A$, is the determinant of a matrix
of forms $G$, with $\partial G=A$. This implies  by 0.3 and 0.4) that the ideal generated
by all the $(k-1)\times (k-1)$ minors of $G$ coincides with $R'$ in degree $T(A')$. 
If we take $G_i=$ the matrix $G$ with the $i$-th row canceled, we thus have an instance
of $k$ matrices with $\partial G_i=A'$, whose minors generate $R'_n$, 
for all $n\geq T(A')$. Thus the Lemma is true for the $k$ matrices coming
from $G$ by erasing  one row at a time.  Thus the Lemma
is true for a general set of $k$ matrices of size $(k-1)\times k$, 
all having degree matrix $A'$.

Now, let us consider the general case.
Since $A'$ is homogeneous and ordered, the $i$-th row of $A'$ is obtained from the last
row of $A'$ by adding a fixed non-negative integer to every entry. We define the
{\it diameter} $d(A')$ of $A'$ to be the (constant) difference between the entries
in the first row of $A'$ and the entries in the last row.

We do induction on $d(A')$, and notice that the case $d(A')=0$ is 
exactly the case where all the rows of $A'$ are equal.

Assume that $d(A')>0$ and the Lemma is true for all matrices with diameter 
smaller than $d(A')$.
Let $m\geq 1$ be the number of rows of $A'$ which are equal to the first row.
Then by subtracting $1$ from the entries in the first $m$ rows of $A'$  
we get a new matrix $A''$ which is still an ordered non-negative integer
homogeneous matrix, with diameter $d(A'')=d(A')-1$.
Then, by induction, the Lemma holds for $A''$, i.e. $M_{T(A'')}((A'')^k)$ holds. 
Notice that $T(A'')=T(A')-(m+1).$

For $j=0,\dots,m$ call $A_j$ the matrix obtained by adding $1$ to the entries in the
first $j$ rows of $A''$. Each $A_j$ is again an ordered non-negative integer 
homogeneous matrix. Moreover $A_0=A''$ and $A_m=A'$ and $T(A_j)=T(A'')+j+1$ for $j>0$.
We will prove by induction that condition
$M_s((A_{j-1})^j,(A_j)^{k-j})$ holds for $j=1,\dots,m$ and $s=T(A'')+j=T(A_j)-1$.

For $j=1$, we prove that $M_{T(A'')}((A'')^k)$ implies
$M_{T(A'')+1}(A_0,A_1^{k-1})$. Since $M_{T(A'')}((A'')^k)$
is true, then if we take $k$ general matrices
$G_1,\dots, G_k$ with $\partial G_i=A''$ for all $i$,
and call $S$ the quotient $S=R'/I_{k-1}(G_1)$, 
then we have that the image of the  ideal $I_{k-1}(G_2)+
 \dots + I_{k-1}(G_k)$ fills $S$ in degree $T(A'')$. Moreover, by Remark \ref{HB},
we know that the multiplication by a general linear form
gives a surjection $S_{T(A'')}\to S_{T(A'')+1}$. For
$i=2,\dots, k$ call $G_i^L$ the matrix obtained by multiplying
the entries of the first row of $G_i$ by a general linear form $L$. 
The $(k-1)\times (k-1)$ minors of each $G^L_i$ are the $(k-1)\times (k-1)$
minors of $G_i$ multiplied by $L$. Thus
 $I_{k-1}(G^L_2)+ \dots + I_{k-1}(G^L_k)$ is equal to 
 $L(I_{k-1}(G_2)+  \dots + I_{k-1}(G_k))$ and therefore its image
 fills $S_{T(A'')+1}$. It follows that $R'$ is equal
 to $I_{k-1}(G_1)+I_{k-1}(G^L_2)+ \dots + I_{k-1}(G^L_k)$ in degree $T(A'')+1$.
 Thus  we have a particular set
 of matrices $G_1,G^L_2,\dots, G^L_k$ with $\partial G_1=A''=A_0$ and
 $\partial (G^L_i)=A_1$, such that all their maximal minors generate
 an ideal which coincides with $R'$ in degree $T(A'')+1$. 
 By semicontinuity, we see that $M_{T(A'')+1}(A_0,A_1^{k-1})$ holds.
 
In an analogous way, for $1< j\leq m$ one proves  that $M_{T(A'')+j-1}(A_{j-2}^{j-1},
A_{j-1}^{k-j+1})$ implies $M_{T(A'')+j}(A_{j-1}^j,A_j^{k-j})$. Namely, take $k$ 
general matrices $G_1,\dots, G_k$ with $\partial G_i=A_{j-2}$ for $i=1,\dots, j-1$ and
$\partial G_i=A_{j-1}$ for $i\geq j$.
Call $S'$ the quotient $S'=R'/I_{k-1}(G_j)$. Now
$M_{T(A'')+j-1}(A_{j-2}^{j-1}, A_1^{k-j+1})$ implies that the image of the  ideal 
$I_{k-1}(G_1)+ \dots + I_{k-1}(G_{j-1})+ 
I_{k-1}(G_{j+1})\dots +I_{k-1}(G_k)$ fills $S'$ in degree $T(A'')+j-1$. Moreover, 
since $T(A_{j-1})=T(A'')+j$,
by Remark \ref{HB}, we know that the multiplication by a general linear form
gives a surjection $S'_{T(A'')+j-1}\to S'_{T(A'')+j}$. For
$i=1,\dots, j-1$ call $G_i^L$ the matrix obtained by multiplying
the entries of the $(j-1)$-th row of $G_i$ by a general linear form $L$. 
For $i=j+1,\dots, k$ call $G_i^L$ the matrix obtained by multiplying
the entries of the $j$-th row of $G_i$ by the same general linear form $L$.
The $(k-1)\times (k-1)$ minors of each $G^L_i$ are the $(k-1)\times (k-1)$
minors of $G_i$ multiplied by $L$. Thus
 $I_{k-1}(G^L_1)+ \dots I_{k-1}(G^L_{j-1})+I_{k-1}(G^L_{j+1})+ \dots +I_{k-1}(G^L_k)$ is equal to 
 $L(I_{k-1}(G_1)+ \dots I_{k-1}(G_{j-1})+I_{k-1}(G_{j+1})+ \dots +I_{k-1}(G_k))$ and therefore its image
 fills $S'_{T(A'')+j}$. It follows that $R'$ is equal
 to $I_{k-1}(G^L_1)+ \dots I_{k-1}(G^L_{j-1})+I_{k-1}(G_j)+I_{k-1}(G^L_{j+1})+ \dots +I_{k-1}(G^L_k)$
 in degree $T(A'')+j$.
Thus  we have a particular set
 of matrices $G^L_1,\dots, G^L_{j-1},G_j, G^L_{j+1},\dots, G^L_k$ 
  such that all their maximal minors generate
 an ideal which coincides with $R'$ in degree $T(A'')+j$.
 Notice that $\partial G^L_i=A_{j-1}$ when $i=1,\dots, j-1$, $\partial G_j=A_{j-1}$
  and $\partial G^L_i=A_j$ for $i>j$.
 Thus, by semicontinuity, we see that $M_{T(A'')+j}(A_{j-1}^j,A_j^{k-j})$ holds.
 
After $m$ steps, we get that $M_{T(A'')}((A'')^k)$ implies $M_{T(A'')+m}(A_{m-1}^m,A_m^{k-m})$. 
It remains to show that $M_{T(A'')+m}(A_{m-1}^m,A_m^{k-m})$ implies 
$M_{T(A')}((A')^k)$.

Take general matrices of forms $G_1,\dots, G_k$ with $\partial G_i=A_{m-1}$ 
for $i=1,\dots, m$ and $\partial G_i=A_m=A'$ for $i\geq m+1$ (recall that $k>m$). 
Call $S''$ the quotient of $R'$ by the ideal
$I_{k-1}(G_{m+1})+\dots +I_{k-1}(G_k)$. Since $S''$ is a quotient of
$R'/I_{k-1}(G_k)$ and $\partial G_k=A_m=A'$,  by Remark \ref{HB}
we know that the multiplication by a general linear form
gives a surjection $S''_{T(A'')+m}\to S''_{T(A'')+m+1}$.  Now
$M_{T(A'')+m}(A_{m-1}^m,A_m^{k-m})$ implies that the image of the  ideal 
$I_{k-1}(G_1)+ \dots + I_{k-1}(G_{m-1})$ fills $S''$ in degree $T(A'')+m$. 
 For $i=1,\dots, m$ call $G_i^L$ the matrix obtained by multiplying
the entries of the $m$-th row of $G_i$ by a general linear form $L$. 
The $(k-1)\times (k-1)$ minors of each $G^L_i$ are the $(k-1)\times (k-1)$
minors of $G_i$, multiplied by $L$. Thus
 $I_{k-1}(G^L_1)+ \dots I_{k-1}(G^L_m)$ is equal to 
 $L(I_{k-1}(G_1)+ \dots I_{k-1}(G_m))$  and therefore its image
 fills $S''$ in degree $T(A'')+m+1=T(A')$. It follows that $R'_{T(A')}$ is equal
 to $I_{k-1}(G^L_1)+ \dots I_{k-1}(G^L_m)+I_{k-1}(G_{m+1})+\dots +I_{k-1}(G_k)$.
Thus  we have a particular set
 of matrices $G^L_1,\dots, G^L_m,G_{m+1},\dots, G_k$ 
  such that all their maximal minors generate
 an ideal which coincides with $R'$ in degree $T(A')$.
 Notice that now $\partial G^L_i=A'$ when $i=1,\dots, m$, and
 also  $\partial G_i=A'$ for $i>m$.
 Thus, by semicontinuity, we see that $M_{T(A')}((A')^k)$ holds.
 \end{proof}  

\begin{ex} We give an explicit description of the previous argument,
for a particular   $3\times 4$ matrix.

Assume we want to know that the maximal minors of $k=4$ general
matrices, with degree matrix
$$ A'=\begin{pmatrix} 5 & 6 & 8 & 9\\ 5 & 6 & 8 & 9 \\ 
2 & 3 & 5 & 6\end{pmatrix}
$$
generate the ring $R'$ in degree $T(A')=25$. I.e. we want to show that
$M_{25}((A')^4)$ holds. 
$A'$ is an ordered homogeneous matrix of non-negative integers,
with diameter $3$. The first two rows of $A'$ are equal so, in the notation of Lemma
\ref{tech}, $m=2$.
Thus in order to decrease the diameter, we need to subtract $1$ from the
first two rows. We obtain, in this way, the matrix
$$
 A''=\begin{pmatrix} 4 & 5 & 7 & 8\\ 4 & 5 & 7 & 8 \\ 
2 & 3 & 5 & 6\end{pmatrix}.
$$
which is still ordered and whose diameter is $2$. Since $T(A'')=22$,
we may assume by induction that $M_{22}((A'')^4)$ holds.

We will need the auxiliary matrix 
$$ A_1=\begin{pmatrix} 5 & 6 & 8 & 9\\ 4 & 5 & 7 & 8\\ 
2 & 3 & 5 & 6\end{pmatrix}\
$$
for which $T(A_1)=24$. Following the proof of the Lemma, we show
that $M_{22}((A'')^4)$ implies $M_{23}((A'')^1,A_1^3)$ which in turn implies
$M_{24}(A_1^2,(A')^2)$, which finally implies $M_{25}((A')^4)$. 

Indeed, take $4$ general matrices $G_1,G_2,G_3,G_4$ whose degree matrix is $A''$.
By $M_{22}((A'')^4)$, we know that $I_3(G_1)+\dots + I_3(G_4)$ fills $R'$
in degree $22$. Moreover the multiplication by a general linear form
gives an isomorphism $(R'/I_3(G_1))_{22}\to (R'/I_3(G_1))_{23}$.
Thus if $G^L_i$, $i=2,3,4$ is the matrix obtained from $G_i$ by multiplying
the first row by a general linear form $L$, we see that $\partial G^L_i=A_1$ and
$I_3(G_1)+I_3(G^L_2)+I_3(G^L_3)+I_3(G^L_4)$ coincides with $R'$ in degree $23$.
Thus, by semicontinuity, $M_{23}((A'')^1,A_1^3)$ holds.
Now take new general matrices $H_1,H_2,H_3,H_4$ with $\partial H_1=A''$ and
$\partial H_2=\partial H_3=\partial H_4=A_1$, so that, by $M_{23}((A'')^1,A_1^3)$, the ideal
$I_3(H_1)+\dots +I_3(H_4)$ coincides with $R'$ in degree $23$.
Since $\partial H_2=A_1$ and $T(A_1)=24$, the multiplication by a general 
linear form determines an isomorphism $(R/I_3(H_2))_{23}\to (R/I_3(H_2))_{24}$.
Take a general linear form $X$. Call $H_1^X$ the matrix obtained by multiplying
the first row of $H_1$ by $X$, and call $H_3^X$ (resp. $H_4^X$) 
the matrix obtained by multiplying the second row of $H_3$ (resp. $H_4$) by $X$.
Since $I_3(H^X_1)+I_3(H_3^X)+I_3(H^X_4)=X(I_3(H_1)+I_3(H_3)+I_3(H_4))$, then
$I_3(H^X_1)+I_3(H_2)+I_3(H_3^X)+I_3(H^X_4)$ coincides with $R'$ in degree $24$.
Notice that $\partial H_1^X=\partial H_2=A_1$ while $\partial H^X_3=\partial H_4^X=A'$.
Thus, by semicontinuity, $M_{24}(A_1^2,(A')^2)$ holds.
Finally, take new general matrices $K_1,K_2,K_3,K_4$ with $\partial K_1=\partial
K_2=A_1$ and $\partial K_3=\partial K_4=A'$, so that, by $M_{24}(A_1^2,(A')^2)$, the ideal
$I_3(K_1)+\dots +I_3(K_4)$ coincides with $R'$ in degree $24$.
Since $\partial K_3=A'$ and $T(A')=24$, the multiplication by a general 
linear form determines an isomorphism $(R'/I_3(K_3))_{24}\to (R'/I_3(K_3))_{25}$
and consequently a surjection $S_{24}\to S_{25}$, where $S=R'/(I_3(K_3)+I_3(K_4))$.
Take a general linear form $Y$. Call  $K_1^Y$ (resp. $K_2^Y$) 
the matrix obtained by multiplying the second row of $K_1$ (resp. $K_2$) by $Y$.
Since $I_3(K^Y_1)+I_3(K_2^Y)=Y(I_3(K_1)+I_3(K_2))$, then
$I_3(K^Y_1)+I_3(K^Y_2)+I_3(K_3)+I_3(K_4)$ coincides with $R'$ in degree $25$.
Notice that $\partial K_1^Y=\partial K^Y_2=\partial K_3=\partial K_4=A'$.
Thus, by semicontinuity, $M_{25}((A')^4)$ holds.
\end{ex}

\section{The proof of the main Theorem}
 
 Let  $R=\CC[x,y,z,t]$.  The technical Lemma in the previous section gives us
  a Lefschetz-type property for certain quotients of $R$.
 
 \begin{prop} \label{Lef}
Let $A'=(a_{ij})$ be a homogeneous $(k-1)\times k$ matrix 
of non--negative integers. 
Let $G_1,\dots , G_k$ be a general choice of matrices of quaternary forms such that
$\partial G_i=A'$ for all $i$. 

If $J$ is the ideal generated by
all the maximal minors of the $G_i$'s, then the multiplication map
$(R/J)_{n-1}\to (R/J)_n$ by a general linear form is surjective 
when $n\geq T(A')$. 
 \end{prop}
\begin{proof}
The proof follows immediately from Lemma \ref{tech}. Indeed,
 since condition $M_{T(A')}((A')^k)$ holds,  the residues of the matrices
$G_i$'s in $R'=R/\langle x\rangle=\CC[y,z,t]$ have maximal minors 
which generate $R'$ in all degrees $n\geq T(A')$. Thus, modulo $J$, every element
in $(R/J)_n$ is divisible by $x$. Hence,  multiplication
by $x$ surjects onto $(R/J)_n$.
\end{proof}

We have all the ingredients to prove the main result,
which we recall here:

\begin{thm}\label{main} 
Let $A$ be a homogeneous $k\times k$ matrix of \emph{non--negative} 
integers, with degree $d$. Then a general form of degree $d$ in $4$
variables is the sum of $k$ determinants of matrices of forms,
with degree matrix $A$. 
\end{thm}
\begin{proof}
Let $A=(a_{ij})$ be a square, homogeneous 
$k\times k$ matrix of non--negative integers.

We will prove the Theorem by induction on the degree (= trace) $d$ of $A$.

If $d=0$, then  all the entries of $A$ are $0$ and the Theorem is obvious.

Assume, by induction, that the Theorem holds for all matrices with trace $< d$ and
assume that $A$ has trace $d$. 
We also assume that $A$ is ordered. Then, since $a_{k1}\geq 0$, $a_{1k}+a_{k1} =
a_{11}+a_{kk}$ and $a_{1k}=\max\{a_{ij}\}$, we see that, 
after reflecting $A$ along its anti-diagonal, if necessary, we may also assume $a_{11}>0$.
This is so because the determinant
of a matrix with degree matrix $A$ is equal to the determinant of
any matrix with degree matrix obtained by reflecting across the anti-diagonal
of $A$.

Let $B$ be the matrix obtained from $A$ by subtracting $1$ from the
first row. Then $B=(b_{ij})$ is again a homogeneous matrix 
of non-negative integers, whose trace is $d-1$.  
Thus the Theorem holds for $B$. Hence, by Proposition \ref{tgkxk}
and Remark \ref{reduc},  for a general choice of $k$ matrices of
quaternary forms $M_1,\dots, M_k$, with $\partial M_i=B$,
the ideal generated by all the $(k-1)\times (k-1)$ minors
of the matrices $M_i$'s coincides with $R$ in degree $d-1$.  
Now, if we forget the first rows of the matrices $M_i$,
we get $k$ matrices $G_1,\dots, G_k$, of size $(k-1)\times k$, whose degree
matrix $B'$ equals $B$ with the first row canceled (which is equal to $A$
with the first row canceled).
As $A$ is ordered, $T(B')$  is at most $d$. Thus, by Proposition \ref{Lef}, all the 
$(k-1)\times (k-1)$ minors of $G_1,\dots, G_k$ 
 generate an ideal $J$ such that multiplication by a general linear form $L$
 determines a surjective map $(R/J)_{d-1}\to (R/J)_d$.
Call $M'_i$ the matrix obtained from
$M_i$ by multiplying the first rows by $L$. It follows that
the ideal generated by the $(k-1)\times (k-1)$ minors
of the matrices $M'_1,\dots, M'_k$'s coincides with $R$ in degree $d$.
By semicontinuity, this last property holds for a general
choice of  $k$ matrices $H_1,\dots, H_k$, with $\partial H_i=A$
for all $i$.  

By Proposition \ref{tgkxk} and Remark \ref{reduc}, the Theorem follows.
\end{proof}

\begin{rem}\label{sharprem}
It is very reasonable to ask when the bound given in Theorem 2.2 is  sharp.

As we observed in the Introduction (immediately after formula \eqref{sharp}), 
the bound is sharp when all the entries of the matrix
$A$ are equal to a number $a$, which is sufficiently large with respect to $k$. 
Standard arithmetic shows that,
indeed, the bound is sharp whenever $\min\{a_{ij}\}\gg k$. In all these cases, 
a general form of degree $d=$ tr$(A)$
cannot be written as a sum of determinants of fewer than $k$ matrices of forms, 
with degree matrix $A$.
Here the word ''general'' means that forms requiring less than $k$ summands are 
contained in a (non-trivial) Zariski closed subset of the space of all forms 
of degree $d$.

On the other hand, we will provide, in the next section, examples of degree 
matrices $A$ (with some small entry) and degrees  $d$
such that fewer than $k$ summands are sufficient for general forms of degree $d$.

The problem of finding the  complete range in which our theorem is sharp seems,
at least technically, rather  laborious.
\end{rem}

\section{Improvements and open questions}\label{impro}

In this section we show how the main theorem can sometimes be improved. 
 We also give some open questions on the subject.
\smallskip

Assume that we are dealing with $3\times 3$ matrices of forms in $4$ variables. 
Then Theorem \ref{main} above states that for any homogeneous
matrix $A$ of degree $d$, with non-negative entries, a general form
of degree $d$ is the sum of three determinants of matrices of forms, whose
degree matrix is $A$.

We want to refine this statement and show that when the minimal entry
of $A$ is $1$, then we can write a general form of degree $d$ as the sum
of determinants of two matrices whose degree matrix is $A$. 

We will get the proof by using the Lefschetz property
of Artinian complete intersection rings.

\begin{thm} Let $A$ be a $3\times 3$ homogeneous matrix of non-negative integers,
of degree $d$, whose minimal entry is $1$. Then a general form of degree $d$
in $4$ variables is the sum of the determinants of two matrices of forms,
whose degree matrix is $A$.
\end{thm}
\begin{proof} Assume that $A=(a_{ij})$ is ordered.
We have $a_{31}=1$. We will prove the statement by induction on the 
biggest entry $a_{13}$ of $A$.

Assume $a_{13}=1$. Then all the entries of $A$ are $1$ and the degree
of $A$ is $3$. It is classical that a general cubic form is the
determinant of a single $3\times 3$ matrix of linear forms, since the corresponding
surface contains a twisted cubic curve (see e.g. \cite{G1855}). Thus the statement
trivially holds in this case.

Assume now $a_{13}>1$, so that $a_{11}+a_{33}=a_{31}+a_{13}>2$.
After taking the reflection along the anti-diagonal, we may assume that $a_{11}>1$,
so that all the entries in the first row are bigger than $1$.
Assume that the statement is true for all matrices with trace smaller than the trace
$d$ of $A$.

Let $B$ be the matrix obtained by erasing the first row of $A$.
By the Hilbert-Burch Theorem, the $3$ minors of a general matrix of forms,
with degree matrix $B$, vanish along an arithmetically Cohen-Macaulay
curve $C$, which is contained in a complete intersection of surfaces of degrees
$u= a_{22}+1=a_{21}+a_{32}$ and $t=1+a_{32}= a_{21}+a_{33}$. 
 Thus, if $F_1,F_2$ are two general matrices of forms,
with degree matrix equal to $B$, then the ideal $J$ generated by the
$2\times 2$ minors of $F_1,F_2$ is a quotient of a complete intersection artinian ideal
generated by forms of degrees $u,u,t,t$. Since $(u+u+t+t-4)/2= a_{22}+a_{32}$
which is at most equal to $a_{22}+a_{33}$, it follows that for $n\geq a_{21}+a_{22}+a_{23}$
the multiplication map by a general linear form $(R/J)_{n-1}\to (R/J)_n$ surjects.

Now, let $A'$ be the $3\times 3$ matrix obtained from $A$ by decreasing
the first row by $1$ and reordering (if necessary). Then $A'$ satisfies the inductive 
hypothesis, for its degree is smaller than $d$. Thus the statement
holds for $A'$. In particular, by Remark \ref{reduc}, 
if we take two general matrices $G_1,G_2$ of forms, with
degree matrix $A'$, then their $2\times 2$ minors generate the ring $R$ in degree $d-1$.
Moreover, if $J$ is the ideal generated by the $2\times 2$ minors obtained after
deleting the first row in both $G_1,G_2$, then the multiplication by a general
linear form $(R/J)_{d-1}\to (R/J)_d$ surjects.
Thus, by multiplying the first row of both $G_1,G_2$ by a general linear
form, we get two matrices whose degree matrix is $A$ and whose $2\times 2$ minors
generate $R$ in degree $d$. The statement follows from Remark \ref{reduc}. 
\end{proof}

\begin{rem}\label{222} For a specific $k\times k$ matrix $A$ containing small 
positive integers one can check directly, 
with the aid of the Computer Algebra package \cite{DGPS}, 
if the $k-1\times k-1$ minors of $k_0<k$ general matrices of
forms, whose degree matrix is
$A$, are sufficient to generate the polynomial ring $R$ in degree equal to
the trace of $A$.

With this procedure, one can prove for instance that a general form of degree $6$ in
$4$ variables is the sum of two determinants of $3\times 3$ matrices of forms, all of whose 
entries have degree $2$.
\end{rem}

We see then that for some specific homogeneous matrices of non-negative integers $A=(a_{ij})$, 
the minimal number $s(A)$ of determinants of $k\times k$ matrices of
forms with fixed degree matrix $A$, which are necessary to write a general form in $4$
variables of degree $=tr(A)$, can be smaller than $k$. We show how one can produce a sharp conjecture 
for $s(A)$, at last when all the entries of $A$ are positive,
by making more precise the construction already outlined in Example \ref{basic}.
\smallskip

The matrix $A'$ obtained from $A$ by erasing the first row is the
degree Hilbert-Burch matrix of arithmetically normal curves in $\PP^3$, which
fill a dense open subset of an irreducible component $Hilb(A')$ of the Hilbert scheme. 
The dimension of $Hilb(A')$ can be computed from the entries of $A'$.
See  \cite{E75}, Theorem 2.

Now one can construct the incidence variety:
$$ Z=\{(F,C): C\in Hilb(A') \mbox{ and $F$ a surface of degree $d$ containing}\ C\}.$$
The fiber of the projection $Z\to Hilb(A')$ over $C$ is equal to $\PP(H^0(I_C(d))$.
These are projective spaces of the same dimension independent of $C$. One can
easily compute this dimension from the  resolution induced by $A'$
 
\begin{equation}\label{sequ}
 0\to \oplus^{k-1} O(-b_j)\to \oplus^k O(-a_i)\to I_C\to 0.
 \end{equation}
  In particular, $Z$ is irreducible
and one can compute the dimension of $Z$ as a function of $d$ and the entries of $A'$.

Call $V(A)$ the closure of the image of the projection of $Z$ to the space
$\PP^{N_d}$ which parametrizes surfaces of degree $d$ in $\PP^3$.
Recall that $N_d=\binom{d+3}d -1$.
$V(A)$ is exactly the closure of the locus  of surfaces of degree
$d=\deg(A)$, containing a curve $C\in Hilb(A')$,
i.e. it is the locus of those surfaces whose equation is the determinant of a single
matrix of forms $G$, with $\partial G =A$.

The closure of the set of forms which are the sum of $s$ determinants of matrices
$G_1,\dots, G_s$ with $\partial G_i=A$ for all $i$, corresponds to the $s$-th
secant variety of $V(A)$.

The expected value for the dimension of the $s$-th secant variety of $V(A)$
is equal to the minimum between $s\dim(V(A))+s-1$ and the dimension of the
whole space $N_d$. In particular, as soon as $s\dim(V(A))+s-1$ is bigger than or
equal to $N_d$, i.e. as soon as 
$$ s\geq \frac{\binom{d+3}3}{\dim(V(A)) +1}\ $$
then one expects the $s$-th secant variety of $V(A)$ to fill $\PP^{N_d}$.
This means that a general form of degree $d$ should be the sum of $s$ determinants
of matrices $G_1,\dots, G_s$ with $\partial G_i=A$ for all $i$.

When the dimension of the $s$-th secant variety of $V(A)$ is different from
the expected value, then $V(A)$ is said to be {\it $s$-defective}.
Thus one should consider the following problem:

\noindent
{\bf Problem.} Are there homogeneous matrices $A$ of non-negative integers
such that the corresponding variety $V(A)$ is defective? Can one classify them?
\medskip

If one believes that $V(A)$ is not defective, then there is a conjecture
for the minimal integer $s(A)$ such that a general form of degree $d=\deg(A)$
can be written as the sum of $s(A)$ determinants
of matrices  with degree matrix $A$.

Of course, the conjectured bound depends on the dimension of $V(A)$.
Clearly, $V(A)$ is irreducible and one can compute its dimension once one knows
the dimension of a general fiber of the projection $Z\to V(A)$.
For a general choice of $C\in Hilb(A')$ we know that $C$ is a smooth curve
(see \cite{S85}). Since $d$ is bigger than the degree of a maximal generator of 
$I_C$ by Bertini we know that a general $F\in V(A)$ is a smooth
surface in $\PP^3$, hence is regular. It follows that the fiber
of $Z\to V(A)$ over a general point $F$ is given by the union of
a finite number of linear systems on $F$, each of them composed
of arithmetically Cohen-Macaulay curves with the same Betti numbers as $C$.
It follows that for $(F,C)\in Z$ general, the dimension of the fiber
equals the dimension of the space of sections of
the normal bundle $N_{C|F}$ of $C$ {\it in $F$}, which is equal to the dimension
of the linear  system $L_C$ on $F$ that contains the divisor $C$.

This last dimension can be  obtained as follows:
take a general surface $F'$ of minimal degree $a$ passing through $C$. The residue $C'$
of $C$ in the intersection $F\cap F'$ is also an arithmetically Cohen-Macaulay
curve. Every curve which is in the
linear system of $C$ is directly linked to $C'$ by a complete intersection
of type $d,a$ with $a<d$. From this we see that the dimension of the linear system $L_C$
equals the dimension of the space of surfaces of degree $a$ passing through $C'$.
This last number can be computed, since one can compute
a minimal resolution for the ideal sheaf of $C'$, via the
mapping cone procedure (see the description on page 4 of \cite{M98}).

Summing up, we obtain a conjectured number for $s(A)$.

\begin{ex} 
Let us compute the conjectured value for $s(A)$ when $A$ is a $k\times k$ matrix of
linear forms.  We will assume that $V(A)$ is not defective. Notice that $\deg(A)=k$ in this case.

Let $C$ be an arithmetically Cohen-Macaulay curve as above.
The minimal resolution of $I_C$ looks like
$$
 0\to  O^{k-1}(-k)\to \oplus^k O(-k+1)\to I_C\to 0
$$
from which one gets $\dim(Hilb(A'))=2k^2-2k$, which is equal to $4\deg(C)$.
Then one easily computes that the dimension of a general fiber of $Z\to V(A)$ is
equal to $h^0(I_C(k))-1= 3k$, so that $\dim(Z)=2k^2+k$. 

We now compute, for $(F,C)$ general in $Z$, the dimension of the linear system
$L_C$ on $F$, which contains $C$.
Consider  the residue $C'$ of $C$ in the intersection $F\cap F'$, where $F'$ is a surface of
minimal degree $k-1$ passing through $C$. $C'$ is an arithmetically Cohen-Macaulay
curve whose resolution, computed via the mapping cone, is equal to
the resolution of $C$. Thus, the space of surfaces of degree $k-1$ through $C'$
has dimension $k-1$. 

We obtain $\dim(V(A))= 2k^2+1$.
\end{ex}

The computation in the previous example yields the following
\medskip

\noindent
{\bf Conjecture.} A general form of degree $k$ in $4$ variables
is the sum of 
$$ s = \lceil \frac k{12} +\frac 12 +\frac{10k}{12k^2+12} \rceil$$
determinants of $k\times k$ matrices of linear forms.
\medskip

We checked this Conjecture, using a computer aided procedure 
with the package \cite{DGPS}, 
for some initial values of $k$.

\begin{rem} Notice that in the case of matrices of linear forms, the conjectured
minimal number of determinants needed for writing a general form
of degree $k$, is always smaller than $k$.
\end{rem}

\begin{ex}\label{3.5}
Let us perform the previous computation for a $3\times 3$  degree matrix $A$ with
all entries equal to $a$. Assuming that $V(A)$ is not defective and following 
the standard procedure we outlined above, we see that
the dimension of $V(A)$ is  
$$\dim(V(A))= \frac{9a^3+54a^2+99a-48}6$$
while the space of forms of degree $3a$ in $\PP^3$ has dimension
$$ \theta(3a)=\frac{27a^3+54a^2+33a}6.$$
In particular, notice that $\theta(3a)=\dim(V(A))$ when $a=1$, 
consistent with the fact that $V(A)$ is the space of all cubic surfaces.
Indeed a general cubic surface contains a twisted cubic curve and consequently
a general cubic form is the determinant of a $3\times 3$ matrix
of linear forms.

When $a=2,\dots,8$, the quotient $\theta(3a)/(\dim(V(A))+1)$ sits between
$1$ and $2$. It could happen that in these cases the general form of
degree $3a$ is the sum of two determinants of $3\times 3$ matrices of
forms of degree $a$. Remark \ref{222} shows that this does indeed happen for $a=2$.
Using the same procedure, we checked that in all the cases
$a=3,\dots,8$, a general form of degree $3a$ is the sum of two
determinants of $3\times 3$ matrices of forms of degree $a$.

For $a>8$, the quotient $\theta(3a)/(\dim(V(A))+1)$ sits between
$2$ and $3$. In these cases, at least $3$ determinants are needed for
obtaining a general form of degree $3a$. However, our Main Theorem
shows that $3$ determinants are always sufficient. Thus, the example shows that
our Main Theorem is sharp.

A similar computation, for the case of a $k\times k$ degree
matrix $A$ with all entries equal to $a$ and with $a\gg k$, shows that the
minimal  number $s(A)$ of determinants required to obtain
a general form of degree $ka$ cannot be smaller than 
$k$ and hence, by our Main Theorem, must be equal to $k$.
\end{ex}

\begin{rem} One might ask: What happens when some entry of the degree matrix $A=(a_{ij})$
is negative?

If $G$ is a matrix of forms with $\partial G=A$ and $a_{ij}<0$
for some $i,j$, then necessarily the corresponding entry $g_{ij}$ of $G$ is
the $0$ polynomial.

Assume that $A$ is ordered and $a_{ii}<0$ for some $i$. Then 
$a_{ij}=a_{ji}=0$ for all $j\leq i$. It follows that any such matrix
of forms $G$ with degree matrix $A$ has a block of zeroes
which touches the main diagonal. Consequently, $det(G)=0$. In
particular, no non-zero forms can be the sum of any number
of determinants of matrices $G$ with $\partial G=A$.

When the $a_{ii}$'s  are all non-negative but still there exists some entry
$a_{ij}<0$ (so that $j<i$ when $A$ is ordered), the question about
the minimal number of determinants of matrices of forms $G$ with $\partial G=A$,
which are necessary to express a general form of degree $\deg(A)$,
is still open.
\end{rem}

\section{Extension to a larger number of variables.}\label{morevar}

When the number of variables increases we can find 
similar results on the number of determinants that one needs
in order to express a general form. Unfortunately the required
number of determinants grows exponentially.

Indeed, as we noted in the introduction, for a fixed $k\times k$ homogeneous
matrix $A$ of non-negative integers, the question amounts to asking
for the minimal $s$ such that for a general choice of  matrices
$G_1,\dots, G_s$ with degree matrix $A$, the ideal $I$ generated
by all the $(k-1)\times (k-1)$ (submaximal) minors of the $G_i$'s 
coincides with the polynomial ring $R=\mathbb C[x_1,\dots, x_n]$ in degree $d=\deg(A)$.

Assume that all the entries of $A$ are equal to $a$, so that $d=ka$. 
Since $A$ has $k^2$ submaximal minors, then 
$$\dim I_d\leq k^2s \dim R_a =\frac{ k^2s a^{n-1}}{(n-1)!} + o(a^{n-1})$$
while the dimension of $R_d$ is $a^{n-1}k^{n-1}/(n-1)!+o(a^{n-1})$.

So, it is immediate to see that, at least when $a$ grows, in order to have
 $I_d=R_d$ then $s$ must be asymptotically equal to $k^{n-3}$.

With a procedure which is similar to the proof of the Main Theorem
(but with a much heavier notation!), and using induction on the number
of variables, we can prove:

\begin{thm}\label{mainGen} 
Let $A$ be a homogeneous $k\times k$ matrix of \emph{non--negative} 
integers, with degree $d$. Then a general form of degree $d$ in $n\geq 3$
variables is the sum of $k^{n-3}$ determinants of matrices of forms,
with degree matrix $A$. 
\end{thm}
\begin{proof} We make induction on the number $n$ of variables.
The case $n=3$ is the main result in \cite{CM12}, while the case
$n=4$ is Theorem \ref{main} above.

Assume the Theorem is true for forms in $n-1$ variables. We show how the
argument of Lemma \ref{tech}, Proposition \ref{Lef} and Theorem \ref{main} can be modified,
to provide a proof of the statement for forms in $n$ variables.

Fix the matrix $A$ and assume it is ordered. Call $d$ the degree of $A$ 
Forgetting the first row of $A$, we obtain a $(k-1)\times k$ non-negative 
integer matrix $A'$, with $T(A')\leq d$.

The first step consists in proving that given a general set of $k^{n-3}$ matrices
of forms $G_1,\dots, G_{k^{n-3}}$ in the ring $R'=\CC[x_1,\dots, x_{n-1}]$
with $n-1$ variables, with $\partial G_i=A'$ for all $i$, the ideal generated
by all the minors of the $G_i$'s coincides with $R'$ in degree $T(A')$.
This is true by the inductive assumption, when all the rows of $A$ are equal.
Indeed, a special instance of the $G_i$'s can be obtained by taking $k^{n-4}$ matrices
$H_1,\dots, H_{k^{n-4}}$ of forms in $R'$, with $\partial H_i=A$ for all $i$,
and taking the $G_i$'s equal to the matrices obtained by erasing one line from
the $H_j$'s, in all possible ways.  When the rows of $A$ are different, we make
induction on the diameter of $A$, exactly as in the proof of Lemma
\ref{tech}, with the unique difference that one passes from condition  
$M_{T(A'')+j-1}(A_{j-2}^{k^{n-4}(j-1)}, A_{j-1}^{k^{n-4}(k-j+1)})$ 
to condition $M_{T(A'')+j}(A_{j-1}^{k^{n-4}j},A_j^{k^{n-4}(k-j)})$
and to do that one fixes a layer of $k^{n-4}$ matrices and multiplies
one row of the remaining matrices by a general linear form.

The second step consists of the observation that now the same argument as in
the proof of Proposition \ref{Lef} shows that for a general choice
of $k^{n-3}$ matrices $M_i$'s of forms in $R=\CC[x_1,\dots, x_n]$,
with size $(k-1)\times k$ and degree matrix $A'$, the ideal $J$ generated
by all the $(k-1)\times (k-1)$ minors of the $M_i$'s has the property
that the multiplication by a general linear form gives a 
surjective map $(R/J)_{d-1}\to (R/J)_d$.

Finally, as in the proof of the Main Theorem, one uses
 induction on the degree $d$ to show that when $B$ is the matrix obtained 
from $A$ by subtracting $1$ from the
first row, then  for a general choice of $k^{n-3}$ matrices $M_1,\dots, M_{k^{n-3}}$ of
 forms in $R$, with $\partial M_i=B$,
the ideal generated by all the $(k-1)\times (k-1)$ minors
of the matrices $M_i$ coincides with $R$ in degree $d-1$.  
Moreover, by forgetting the first rows of the matrices $M_i$,
we get $k$ matrices $G_1,\dots, G_k$, of size $(k-1)\times k$, whose degree
matrix $B'$ satisfies  $T(B')\leq d$. Thus, by the surjectivity proved above, all the 
$(k-1)\times (k-1)$ minors of $G_1,\dots, G_k$ 
 generate an ideal $J$ such that the multiplication by a general linear form $L$
 determines a surjective map $(R/J)_{d-1}\to (R/J)_d$.
Then, call $M'_i$ the matrix obtained from
$M_i$ by multiplying the first row by $L$. It follows that
the ideal generated by the $(k-1)\times (k-1)$ minors
of the matrices $M'_1,\dots, M'_{k^{n-3}}$ coincides with $R$ in degree $d$.
By semicontinuity, this last property holds for a general
choice of  $k^{n-3}$ matrices $H_i$, with $\partial H_i=A$ for all $i$.  

The Theorem follows by Proposition \ref{tgkxk} and Remark \ref{reduc}.
\end{proof}

\begin{rem} As we observed in the statement of the previous Theorem,
$k^{n-3}$ is almost always a sharp bound.

With an argument analogous to the discussion in Example \ref{3.5} 
one can show that the bound  of Theorem \ref{mainGen} is sharp when the 
$k\times k$ degree matrix has all entries equal to a positive integer $a$ and $a\gg k$.
In these cases, forms that can be written as the sum of fewer  than 
$k^{n-3}$ determinants are contained in a (non-trivial)
Zariski closed subset of the space of all forms of degree $d$.\par

For example, when $A$ is a $4\times 4$ matrix with entries equal to $a\gg 4$, working
 in five variables, it follows that
one needs $16$ determinants in order to express a general form of degree $4a$.
\end{rem}


\begin{thebibliography}{999}



\bibitem[CCG08]{CCG08}
E.~Carlini, L.~Chiantini and A.V. Geramita.
\newblock {\em Complete intersections on general hypersurfaces.}
\newblock {Michigan Math. J.} 57 (2008), 121--136.


\bibitem[CG13]{CG12}
L.~Chiantini and A.V. Geramita.
\newblock {\em On the Determinantal Representation of Quaternary Forms.}
\newblock to appear - Comm. in Alg.


\bibitem[CM12]{CM12} L.~Chiantini and J. Migliore
\newblock {\em Determinantal representation  and subschemes of general plane curves.}
\newblock {Lin. Alg. Applic.} 436 (2012), 1001--1013.

\bibitem[CGO88]{CGO88}
C. Ciliberto, A.V. Geramita and F. Orecchia.
\newblock {\em Remarks on a theorem of Hilbert-Burch.}
\newblock{Boll. U.M.I.} 7 (1988), 463--483.

\bibitem[D21]{D21}  L.E. Dickson.
\newblock{\em Determination of all general homogeneous polynomials expressible 
as determinants with linear elements.}
\newblock{Trans. Amer. Math. Soc.} 22 (1921),  167--179.


\bibitem[E75]{E75}  G. Ellingsrud.
\newblock{\em Sur le schema de Hilbert des varietes de codimension $2$ dans $\PP^e$
a cone de Cohen-Macaulay.}
\newblock{Ann. Sci. Ec. Norm. Sup.} 8 (1975),  423--431.

\bibitem[G1855]{G1855}  H. Grassmann.
\newblock{\em Die stereometrischen Gleichungen dritten grades, und die
dadurch erzeugten Oberfl\" achen.}
\newblock{J. Reine Angew. Math.} 49 (1855),  47--65. 


\bibitem[HL12]{HL12} D. Henrion and J. B. Lasserre. 
\newblock{\em Inner approximations for polynomial matrix inequalities and robust stability regions.}
\newblock{IEEE Trans. Autom. Control.} 57 (2012), 1456--1467.


\bibitem[M98]{M98} J. Migliore. 
\newblock{Introduction to Liaison Theory and Deficiency Modules.}
\newblock{Birkhauser 1998}. 

\bibitem[S85]{S85}  T. Sauer.
\newblock{\em Smoothing projectively Cohen-Macaulay space curves.}
\newblock{Math. Ann.} 272 (1985),  83--90.

 
\bibitem[DGPS11]{DGPS}
W. Decker, G.-M. Greuel, G. Pfister, H. Sch{\"o}nemann. 
\newblock {\sc Singular} {3-1-3} --- {A} computer algebra system for polynomial computations.
\newblock {http://www.singular.uni-kl.de} (2011).
 

\bibitem[V89]{V89}
V. Vinnikov.
\newblock {\em  Complete description of determinantal representations 
of smooth irreducible curves.}  
\newblock{Linear Algebra Appl.} 125 (1989) 103-140.


\end{thebibliography}
\end{document}